\documentclass[reqno]{amsart}
\usepackage{amsmath, amsthm, amssymb, amstext}

\usepackage[left=2.5cm,right=2.5cm,top=2cm,bottom=2cm,includeheadfoot]{geometry}
\usepackage{hyperref,xcolor}
\hypersetup{
	pdfborder={0 0 0},
	colorlinks,
}
\usepackage{enumitem}
\allowdisplaybreaks
\newtheorem{theorem}{Theorem}

\newtheorem{proposition}[theorem]{Proposition}

\DeclareMathOperator*{\loc}{loc}
\newcommand*\diff{\mathrm{d}}

\DeclareMathOperator*{\Ss}{S}

\newcommand{\Lp}[1]{L^{#1}(\Omega)}

\newcommand{\Wp}[1]{W^{1,#1}(\Omega)}

\newcommand{\Wpzero}[1]{W^{1,#1}_0(\Omega)}

\newcommand{\eps}{\varepsilon}
\newcommand{\ph}{\varphi}

\newcommand{\into}{\int_{\Omega}}

\newcommand{\weak}{\rightharpoonup}

\newcommand{\Linf}{L^{\infty}(\Omega)}
\newcommand{\close}{\overline{\Omega}}

\numberwithin{theorem}{section}
\numberwithin{equation}{section}

\newcommand{\R}{{\mathbb R}}

\newcommand{\ran}{\rangle}
\newcommand{\lan}{\langle}

\newcommand{\N}{{\mathbb N}}

\newcommand{\Assg}[1]{\textup{(g)}}

\title[Kirchhoff double phase problems with variable exponents]
{Infinitely many solutions to Kirchhoff double phase problems with variable exponents}

\author[K. Ho]{Ky Ho}
\address[K. Ho]{Department of Mathematical Sciences, Ulsan National Institute of Science and Technology, 50 UNIST-gil, Eonyang-eup, Ulju-gun, Ulsan 44919, Republic of Korea.}
\email{kyho@unist.ac.kr}

\author[P. Winkert]{Patrick Winkert}
\address[P. Winkert]{Technische Universit\"{a}t Berlin, Institut f\"{u}r Mathematik, Stra\ss e des 17.\,Juni 136, 10623 Berlin, Germany}
\email{winkert@math.tu-berlin.de}

\subjclass{35A15, 35J15, 35J60, 35J62}
\keywords{Double phase operator, Kirchhoff term, multiple solutions, variable exponents, variational methods.}

\begin{document}

\begin{abstract}
	In this work we deal with elliptic equations driven by the variable exponent double phase operator with a Kirchhoff term and a right-hand side that is just locally defined in terms of very mild assumptions. Based on an abstract critical point result of Kajikiya \cite{Kajikiya-2005} and recent a priori bounds for generalized double phase problems by the authors \cite{Ho-Winkert-2022}, we prove the existence of a sequence of nontrivial solutions whose $L^\infty$-norms converge to zero.
\end{abstract}

\maketitle

\section{Introduction}

In this paper we study multiplicity results for the following Kirchhoff-type problem
\begin{equation}\label{problem}
		-M\left(\int_\Omega \mathcal{A}(x,\nabla u)\,\diff x\right)\operatorname{div}A(x,\nabla u)= f(x,u)\quad\text{in } \Omega,
		\qquad u= 0 \quad \text{on } \partial\Omega,
\end{equation}
where $\Omega $ is a bounded domain in $\mathbb{R}^{N}$ with a Lipschitz boundary $\partial \Omega$, $\mathcal{A}\colon \Omega\times \mathbb{R}^N \to \mathbb{R}$ and $A\colon \Omega\times \mathbb{R}^N \to \mathbb{R}^N$ are given by
\begin{align*}
	\mathcal{A}(x,\xi)&:=\frac{1}{p(x)}|\xi|^{p(x)}+\frac{\mu(x)}{q(x)}|\xi|^{q(x)},
	\qquad
	A(x,\xi):=\nabla_\xi \mathcal{A}(x,\xi)=|\xi|^{p(x)-2}\xi+\mu(x)|\xi|^{q(x)-2}\xi.
\end{align*}
In the following, for $h \in C(\close)$ we denote $h^{-}:=\inf_{x\in \close} h(x)$ and $h^{+}:=\sup_{x\in \close} h(x)$.

We suppose the subsequent hypotheses:
\begin{enumerate}[label=\textnormal{(H$_1$)},ref=\textnormal{H$_1$}]
	\item\label{H1}
		$p,q\in C^{0,1}(\close)$ such that $1<p(x)<q(x)<N$ for all $x\in\close$, $\left(\frac{q}{p}\right)^+<1+\frac{1}{N}$ and $0 \leq \mu(\cdot)\in C^{0,1}(\close)$.
\end{enumerate}
\begin{enumerate}[label=\textnormal{(H$_2$)},ref=\textnormal{H$_2$}]
	\item\label{H2}
		$M\colon [0, \infty) \to \R$ is a function and $f\colon\Omega\times \R \to \R$ is a Carath\'eodory function such that the following conditions are satisfied:
		\begin{enumerate}[label=\textnormal{(\roman*)},ref=\textnormal{\roman*}]
			\item
				there exist positive constants $t_0,m_0$ such that $M\in C[0,t_0]$ and $m_0\leq M(t)\leq M(t_0)$ for all $t\in [0,t_0]$;
			\item\label{H2ii}
				there exists $\eps_0>0$ such that $f\colon \Omega\times [-\eps_0,\eps_0]\to\R$ is odd with respect to the second variable and $\sup_{|t|\leq\eps_0}\, |f(\cdot,t)|\in L^{\infty}(\Omega)$;
			\item\label{H2iii}
				there exists a nonempty open ball $B\subset\Omega$ such that
				\begin{align*}
					\lim_{t\to 0}\frac{F(x,t)}{|t|^{p_B^-}}=\infty  \quad\text{uniformly for a.\,a.\,}  x\in B,
				\end{align*}
				where $F(x,t):=\int_0^t f(x,\tau)\,\diff \tau$ and $p_B^-:=\inf_{x\in B} p(x)$.
	\end{enumerate}
\end{enumerate}

We shall look for solutions to problem~\eqref{problem} in the Musielak-Orlicz Sobolev space $\left(\Wpzero{\mathcal{H}},\|\cdot\|\right)$, where  $\mathcal{H}(x,t):=t^{p(x)} +\mu(x)t^{q(x)}$ for all $(x,t)\in \close \times [0,\infty)$ (see Section~\ref{Section-2} for the definitions). We call a function $u\in W_0^{1,\mathcal{H}}(\Omega)$ a solution of problem \eqref{problem} if $f(\cdot,u)\in L_{\loc}^1(\Omega)$ and if
\begin{align*}
	M\left(\int_\Omega \mathcal{A}(x,\nabla u)\,\diff x\right)\int_\Omega A(x,\nabla u)\cdot\nabla v\,\diff x=\int_\Omega f(x,u)v\,\diff x
\end{align*}
is satisfied for all $v\in C_c^\infty(\Omega)$.

Our main result reads as follows.

\begin{theorem}\label{maintheorem}
	Let hypotheses \eqref{H1} and \eqref{H2} be satisfied. Then, problem \eqref{problem} admits a sequence of solutions $\{u_n\}_{n\in\N}$ with $\|u_n\|+\|u_n\|_{\infty }\to 0$ as $n\to\infty$, where $\|\cdot\|_{\infty}$ is the norm in $\Linf$.
\end{theorem}

The proof of Theorem \ref{maintheorem} is based on an abstract critical point result of Kajikiya \cite{Kajikiya-2005} (see also Theorem \ref{Kaj.2005}) and recent a priori bounds for generalized double phase problems by the authors \cite{Ho-Winkert-2022} in which new embedding results of the form $W^{1,\mathcal{H}}(\Omega) \hookrightarrow  L^{\Psi}(\Omega)$, with
\begin{align*}
	\Psi(x,t):=t^{r(x)}+\mu(x)^{\frac{s(x)}{q(x)}}t^{s(x)}
	\quad\text{for } (x,t)\in \overline{\Omega}\times [0,\infty),
\end{align*}
where $r,s\in C(\close)$ satisfy $1<r(x)\leq \frac{Np(x)}{N-p(x)}=:p^*(x)$ and $1<s(x)\leq \frac{Nq(x)}{N-q(x)}=:q^*(x)$ for all $x\in\overline{\Omega}$ are presented.

The novelty of our work is the fact that we combine the variable exponent double phase operator with a Kirchhoff term and a reaction term that are both locally defined. As far as we know, there is no other work dealing with a Kirchhoff term along with the variable exponent double phase operator. In case the exponents $p,q$ are constants, we refer to the work of Fiscella-Pinamonti \cite{Fiscella-Pinamonti-2020} who considered the problem
\begin{equation}\label{problem2}
		-m \left[\int_\Omega \left( \frac{|\nabla u|^p}{p} + a(x) \frac{|\nabla u|^q}{q}\right)\,\diff x\right]\mathcal{L}_{p,q}^{a}(u) =f(x,u) \quad \text{in } \Omega,
		\qquad u  = 0 \quad \text{on } \partial\Omega,
\end{equation}
where $f\colon\Omega\times\R\to\R$ is a Carath\'eodory function that satisfies subcritical growth and the Ambrosetti-Rabinowitz condition and
\begin{align}\label{operator_double_phase}
	\mathcal{L}_{p,q}^{a}(u):= \operatorname{div} \left(|\nabla u|^{p-2}\nabla u + a(x) |\nabla u|^{q-2}\nabla u \right), \quad u\in W^{1,\mathcal{H}}_0(\Omega).
\end{align}
By applying the mountain-pass theorem, the existence of a nontrivial weak solution of \eqref{problem2} is shown. Recently, Arora-Fiscella-Mukherjee-Winkert \cite{ Arora-Fiscella-Mukherjee-Winkert-2022a} studied singular Kirchhoff double phase problems given by
\begin{equation*}
		-m \left[\int_\Omega \left( \frac{|\nabla u|^p}{p} + a(x) \frac{|\nabla u|^q}{q}\right)\,\diff x\right]\mathcal{L}_{p,q}^{a}(u) = \lambda u^{-\gamma} +u^{r-1} \quad \text{in } \Omega,
		\qquad u  = 0 \quad \text{on } \partial\Omega,
\end{equation*}
with $\mathcal{L}_{p,q}^{a}$ as in \eqref{operator_double_phase},
where  a suitable Nehari manifold decomposition provides the existence of two different solutions. Another interesting work in the context of Kirchhoff constant exponent double phase problems has been published in \cite{Fiscella-Marino-Pinamonti-Verzellesi-2022} with nonlinear boundary condition based on variational tools. All these works use different methods than in our paper.

It should be noted that the occurrence of a nonlocal Kirchhoff term was first introduced by Kirchhoff \cite{Kirchhoff-1876}. Such problems have a strong background in several applications in physics. Existence results on degenerate and nondegenerate Kirchhoff problems for different type of problems can be found, for example, in the works \cite{Autuori-Pucci-Salvatori-2010, Fiscella-2019, Fiscella-Valdinoci-2014, Mingqi-Radulescu-Zhang-2019,  Pucci-Xiang-Zhang-2015,  Xiang-Zhang-Radulescu-2016} and the references therein.

If $m(t) \equiv 1$ for all $t \geq 0$, problem \eqref{problem} reduces to a double phase problem with variable exponents. In this case, only few and very recent results exist. We refer to the papers \cite{Bahrouni-Radulescu-Winkert-2020, Crespo-Blanco-Gasinski-Harjulehto-Winkert-2022, Kim-Kim-Oh-Zeng-2022,  Leonardoi-Papageorgiou-2022,  Vetro-Winkert-2023,  Zeng-Radulescu-Winkert-2021}, see also the references therein. If $p$ and $q$ are constants, we point out that the double phase operator in \eqref{problem} is associated to the functional
\begin{align}\label{integral_minimizer}
	u\mapsto \int_\Omega \left(\frac{1}{p}|\nabla  u|^{p}+\frac{\mu(x)}{q}|\nabla  u|^{q}\right)\,\diff x,
\end{align}
which occurred for the first time in the work of Zhikov \cite{Zhikov-1986}. Such functionals are used to describe models for strongly anisotropic materials in the context of homogenization and elasticity. In the past decade, functionals of the form \eqref{integral_minimizer} have been studied by several authors concerning regularity properties of local minimizers, we refer to the papers \cite{Baroni-Colombo-Mingione-2015,Baroni-Colombo-Mingione-2018, Colombo-Mingione-2015a, Colombo-Mingione-2015b}, see also \cite{Ragusa-Tachikawa-2020} for variable exponents and the recent paper \cite{DeFilippis-Mingione-ARMA-2021} about nonautonomous integrals.

\section{Preliminaries and Notations}\label{Section-2}

In this section we recall the main properties about  Musielak-Orlicz Sobolev spaces and the double phase operator with variable exponents along with an abstract critical point result.

To this end, let $\Omega$ be a bounded domain in $\mathbb{R}^N$ with Lipschitz boundary $\partial\Omega$ and let $M(\Omega)$ be the space of all measurable functions $u\colon \Omega\to\R$. We denote by $\Lp{r}$ the usual Lebesgue space endowed with the norm $\|\cdot\|_{r}$ for any $1\leq r \leq \infty$. Suppose \eqref{H1} and let $\mathcal{H} \colon \overline{\Omega} \times [0,\infty) \to [0,\infty)$ be the nonlinear function defined by
\begin{align*}
	\mathcal{H}(x,t):=t^{p(x)} +\mu(x)t^{q(x)} \quad\text{for all } (x,t)\in \overline{\Omega} \times [0,\infty).
\end{align*}
The corresponding modular to $\mathcal{H}$ is given by
\begin{align*}
	\rho_{\mathcal{H}}(u) = \into \mathcal{H} (x,|u|)\,\diff x
	=\into \left(|u|^{p(x)}+\mu(x)|u|^{q(x)}\right)\,\diff x
\end{align*}
and the associated Musielak-Orlicz space $\Lp{\mathcal{H}}$ is then defined by
\begin{align*}
	L^{\mathcal{H}}(\Omega)=\left \{u \in M(\Omega) \,:\,\rho_{\mathcal{H}}(u) < +\infty \right \}
\end{align*}
endowed with the Luxemburg norm $\|u\|_{\mathcal{H}} = \inf \left \{ \tau >0 : \rho_{\mathcal{H}}\left(\frac{u}{\tau}\right) \leq 1  \right \}$.
When $\mu(\cdot)\equiv 0$, we write $L^{p(\cdot)}(\Omega)$ in place of $L^{\mathcal{H}}(\Omega)$. Similarly, the Musielak-Orlicz Sobolev space $\Wp{\mathcal{H}}$ is defined by
\begin{align*}
	\Wp{\mathcal{H}}
	=\left \{u \in L^{\mathcal{H}}(\Omega) \,:\,|\nabla u| \in L^{\mathcal{H}}(\Omega) \right \}
\end{align*}
equipped with the norm $\|u\|_{1,\mathcal{H}} = \|u\|_{\mathcal{H}}+\|\nabla u\|_{\mathcal{H}}$,
where $\|\nabla u\|_{\mathcal{H}}=\| \, |\nabla u| \,\|_{\mathcal{H}}$. Moreover, $\Wpzero{\mathcal{H}}$ is the completion of $C^\infty_0(\Omega)$ in $\Wp{\mathcal{H}}$. We know that  $\Lp{\mathcal{H}}$, $\Wp{\mathcal{H}}$ and $\Wpzero{\mathcal{H}}$ are reflexive Banach spaces and we can equip $\Wpzero{\mathcal{H}}$ with the equivalent norm $\|\cdot\|:=\|\nabla \cdot \|_{\mathcal{H}}$,
see \cite{Crespo-Blanco-Gasinski-Harjulehto-Winkert-2022}.

Moreover, we have
\begin{align}
		\|u\|^{p^-}-1 \leq \rho_\mathcal{H}(|\nabla u|) &\leq \|u\|^{q^+}+1 \quad \text{for all }u\in\Wpzero{\mathcal{H}},\label{estimate-modular}\\
		\|u\|^{q^+} \leq \rho_\mathcal{H}(|\nabla u|)&\leq 	\|u\|^{p^-} \quad\text{for all }u\in\Wpzero{\mathcal{H}} \text{ with }\|u\|<1,\label{estimate-modular2}
\end{align}
and
\begin{align}\label{compact-embedding}
	\Wpzero{\mathcal{H}} \hookrightarrow \Lp{r(\cdot)}
\end{align}
is compact for $r \in C(\close)$ with $ 1 \leq r(x) < p^*(x)$ for all $x\in \close$,
see \cite[Propositions 2.13 and 2.16]{Crespo-Blanco-Gasinski-Harjulehto-Winkert-2022}.

Let $B\colon \Wpzero{\mathcal{H}}\to \Wpzero{\mathcal{H}}^*$ be the nonlinear map defined by
\begin{align}\label{operator_representation}
	\langle B(u),v\rangle :=\into \big(|\nabla u|^{p(x)-2}\nabla u+\mu(x)|\nabla u|^{q(x)-2}\nabla u \big)\cdot\nabla v \,\diff x
\end{align}
for all $u,v\in\Wpzero{\mathcal{H}}$, where $\lan\,\cdot\,,\,\cdot\,\ran$ is the duality pairing between $\Wpzero{\mathcal{H}}$ and its dual space $\Wpzero{\mathcal{H}}^*$.  The operator $B\colon \Wpzero{\mathcal{H}}\to \Wpzero{\mathcal{H}}^*$ has the following properties, see \cite[Theorem 3.3]{Crespo-Blanco-Gasinski-Harjulehto-Winkert-2022}.

\begin{proposition}\label{properties_operator_double_phase}
	Let hypotheses \eqref{H1}  be satisfied. Then, the operator $B$ defined in \eqref{operator_representation} is bounded, continuous, strictly monotone and of type $(\Ss_+)$, that is, $u_n\weak u$  in $\Wpzero{\mathcal{H}}$ and $\limsup_{n\to\infty}\,\langle Bu_n,u_n-u\rangle\le 0$, imply $u_n\to u$ in $\Wpzero{\mathcal{H}}$.
\end{proposition}

Let $X$ be a Banach space, let $X^*$ be its dual space and let $\ph\in C^1(X,\R)$. We say that $\{u_n\}_{n\in\mathbb N}\subseteq X$ is a Palais-Smale sequence (\textnormal{(PS)}-sequence for short) for $\ph$ if $\{\ph(u_n)\}_{n \in \mathbb{N}}\subseteq\R$ is bounded and $\ph'(u_n)\to 0 \quad\text{in }X^*\quad\text{as }n\to\infty$. We say that $\ph$ satisfies the Palais-Smale condition (\textnormal{(PS)}-condition for short) if any \textnormal{(PS)}-sequence $\{u_n\}_{n\in\N}$ of $\ph$ admits a convergent subsequence in $X$.
The proof of Theorem \ref{maintheorem} is based on the following abstract critical point result due to Kajikiya \cite[Theorem 1]{Kajikiya-2005}.

\begin{theorem}\label{Kaj.2005}
	Let $(X,\|\cdot\|)$ be an infinite dimensional Banach space and  $J\in C^1(X,\R)$ such that the following two assumptions hold:
	\begin{enumerate}
		\item[\textnormal{(J1)}]
			$J$ is even, bounded from below, $J(0)=0$ and it satisfies the \textnormal{(PS)}-condition.
		\item[\textnormal{(J2)}]
			For any $k \in \N$, there exist a $k$-dimensional subspace $X_k$ of X and a number $r_k>0$ such that $\sup_{X_k \cap S_{r_k}}J(u) <0$, where $S_{r_k} =\{ u \in X\,:\,\, \|u\|=r_k\}$.
	\end{enumerate}
	Then, the functional $J$ admits a sequence of critical points $\left\{v_k\right\}_{k\in\N}$ satisfying  $\|v_k\| \to 0$ as $k \to \infty$.
\end{theorem}

\section{Proof of the main result}
In this section we give the proof of Theorem \ref{maintheorem}.

\begin{proof}[Proof of Theorem \ref{maintheorem}]
	Since the the conditions on the Kirchhoff and reaction terms are given locally, the corresponding energy functional associated with problem \eqref{problem} may not be well defined. In order to deal with this difficulty and get the symmetry of the associated energy functional, we first modify the functions $M$ and $f$ as follows: We define  $M_0, \widehat{M}_0\colon [0,\infty) \to \R$ given by
	\begin{equation*}
		M_0(t):=
		\begin{cases}
			M(t) &\text{if }0\leq t\leq t_0,\\
			M(t_0) &\text{if }t>t_0,
		\end{cases}
		\quad\text{and}\quad
		\widehat{M}_0(t):=\int_0^t M_0(s)\,\diff s.
	\end{equation*}
	It is clear that $M_0\in C([0,\infty),\R)$ and
	\begin{align}
		m_0&\leq M_0(t)\leq M(t_0) \quad \text{for all }t\in[0,\infty),\label{Est.M0}\\
		m_0t&\leq \widehat{M}_0(t)\leq M(t_0)t\quad \text{for all }t\in[0,\infty).\label{Est.hat.M0}
	\end{align}
	Next, let $\eta \in C_c^\infty(\R)$ be a function such that $0 \leq \eta(t)=\eta(-t) \leq 1$ for  $t\in\R$ and
	\begin{align*}
		\eta(t)=1 \quad\text{for }|t|\leq \frac{\eps_0}{2},\qquad \eta(t)=0\quad \text{for }|t| \geq \eps_0,
	\end{align*}
	where $\eps_0$ is given in \eqref{H2}\eqref{H2ii}.
	For $x\in \Omega$ we define
	\begin{align*}
		h(x,t):=
		\begin{cases}
			\eta(t)f(x,t)&\text{if } |t|\leq\eps_0,\\
			0& \text{if }|t|\geq\eps_0,
		\end{cases}
		\quad \text{and} \quad
		H(x,t):=\int_0^t h(x,s)\,\diff s.
	\end{align*}
	Obviously, we have
	\begin{equation}\label{PT.Sub.h}
		\sup_{t\in\mathbb{R}}\, |h(x,t)|\leq \sup_{|t|\leq\eps_0}\, |f(x,t)|=:f_0(x)\quad \text{for a.\,a.\,} x\in\Omega.
	\end{equation}
	Furthermore, $H$ is even with respect to the second variable and
	\begin{equation}\label{PT.Sub.H}
		\sup_{t\in\mathbb{R}}\, |H(x,t)|\leq \eps_0 f_0(x)
		\quad \text{for a.\,a.\,}\ x\in\Omega.
	\end{equation}
	Note that $f_0\in L^{\infty}(\Omega)$ by hypothesis \eqref{H2}\eqref{H2ii}.

	Now, we consider the following modified problem
	\begin{equation}\label{problem'}
			-M_0\left(\int_\Omega \mathcal{A}(x,\nabla u)\,\diff x\right)\operatorname{div}A(x,\nabla u)= h(x,u)\quad \text{in } \Omega,
			\quad u= 0 \quad \text{on } \partial\Omega.
	\end{equation}
	We point out that if $\{v_k\}_{k\in\N }$ is a sequence of solutions to problem \eqref{problem'} satisfying $\|v_k\|+\|v_k\|_{\infty}\to 0$ as $n\to\infty$, then $\{v_k\}_{k\geq k_0}$ is a sequence of solutions to problem \eqref{problem} for some $k_0\in\N$. In order to derive the desired conclusion, we will apply Theorem \ref{Kaj.2005} for $(X,\|\cdot\|):=(W_0^{1,\mathcal{H}}(\Omega),\|\nabla\cdot \|_{\mathcal{H}})$ and
	\begin{align*}
		J(u) :=\widehat{M}_0\left(\int_\Omega \mathcal{A}(x,\nabla u)\,\diff x\right)-\int_\Omega H(x,u)\,\diff x,\quad u\in X.
	\end{align*}
	First, we see that $J\colon X\to\R$ is of class $C^1$ and its Fr\'echet derivative $J'\colon X\to X^*$ is given by
	\begin{align*}
	\left\langle J' (u),v\right\rangle= M_0\left(\int_\Omega \mathcal{A}(x,\nabla u)\,\diff x\right)\int_\Omega A(x,\nabla u)\cdot\nabla v\,\diff x-\int_\Omega h(x,u)v\,\diff x
	\end{align*}
	for all $u,v\in X$. 
	 Clearly, any critical point of $J$ is a solution of problem \eqref{problem'}. We will verify that $J$  satisfies conditions \textnormal{(J1)} and \textnormal{(J2)} of Theorem \ref{Kaj.2005}.

	{\bf Step 1:} $J$ fulfills \textnormal{(J1)}

	Clearly, $J$ is even and $J(0)=0$. By \eqref{Est.hat.M0}, \eqref{PT.Sub.H} and \eqref{estimate-modular}, we have
	\begin{align*}
		J(u)\geq m_0\int_\Omega \mathcal{A}(x,\nabla u)\,\diff x-\eps_0\int_\Omega f_0(x)\,\diff x
		\geq \frac{1}{q^+}\left(\|u\|^{p^-}-1\right)-\eps_0\|f_0\|_{1}
		\quad \text{for all }u\in X.
	\end{align*}
	This implies that $J$ is coercive and bounded from below on $X$ since $p^->1$. For verification of the \textnormal{(PS)}-condition for $J$, let $\{u_n\}_{n\in\N}$ be a \textnormal{(PS)}-sequence for $J$, that is
	\begin{equation}\label{A.J'}
		J'(u_n)\to 0\quad  \text{in} \ X^*
	\end{equation}
	and
	\begin{equation}\label{A.bound-J}
		\sup_{n\in\mathbb{N}}|J(u_n)|< \infty.
	\end{equation}
	Then, the coercivity of $J$ and \eqref{A.bound-J} guarantee the boundedness of $\{u_n\}_{n\in\N}$ in $X$. Thus, up to a subsequence if necessary, we have
	\begin{equation}\label{PT.maintheorem.convergence}
		u_n \rightharpoonup u  \quad \text{in }  X
		\quad \text{and}\quad
		u_n\to u\quad \text{in } L^{1}(\Omega),
	\end{equation}
	by \eqref{compact-embedding}. On the other hand, we have
	\begin{align*}
		&M_0\left(\int_\Omega \mathcal{A}(x,\nabla u_n)\,\diff x\right)\int_{\Omega}A(x,\nabla u_n)\cdot (\nabla u_n-\nabla u)\,\diff x
		=\big\langle J'(u_n) ,u_n-u \big\rangle +\int_{\Omega}h(x,u_n)(u_n-u)\,\diff x.
	\end{align*}
	Combining this with \eqref{Est.M0} and \eqref{PT.Sub.h} yields
	\begin{align*}
		m_0\left|\int_{\Omega}A(x,\nabla u_n)\cdot (\nabla u_n-\nabla u)\,\diff x\right|\leq \|J'(u_n)\|_{X^*} \|u_n-u\|+\eps_0\|f_0\|_{\infty} \|u_n-u\|_{1}.
	\end{align*}
	Invoking \eqref{A.J'}, \eqref{PT.maintheorem.convergence} and the boundedness of $\{u_n\}_{n\in\N}$ in $X$, from the last inequality it follows that
	\begin{align*}
		\int_{\Omega}A(x,\nabla u_n)\cdot (\nabla u_n-\nabla u)\diff x\to 0\quad \text{as } n\to\infty.
	\end{align*}
	Hence, $u_n\to u$ in $X$ in view of Proposition \ref{properties_operator_double_phase}. Thus,  $J$ satisfies the \textnormal{(PS)}-condition and so \textnormal{(J1)} is fulfilled.

	{\bf Step 2:} $J$ fulfills \textnormal{(J2)}

	Let $k\in\mathbb{N}$ be given and set $X_k:=\operatorname{span}\{\varphi_1,\varphi_2,\cdots,\varphi_k\}$, where $\varphi_n$ is an eigenfunction corresponding to the $n$-th eigenvalue of the eigenvalue problem
	$-\Delta u=\lambda u$ in $B$, $u=0$  on $\partial B$, and it is extended on $\Omega$ by putting $\varphi_n(x)=0$ for $x\in\Omega\setminus B$. Since $X_k$ is finitely dimensional, all norms on $X_k$ are equivalent. Thus, we find positive constants $\alpha_k,\beta_k$ such that
	\begin{equation}\label{d2}
		\beta_k\|u\|_{\infty}
		\leq\|u\|\leq
		\alpha_k\|u\|_{p_B^-}
		\quad \text{for all }u \in X_k.
	\end{equation}
	By condition \eqref{H2}\eqref{H2iii} we can choose
	\begin{align}\label{assumptions}
		M_k>\frac{M(t_0)\alpha_k^{p_B^-}}{p^-}
		\quad\text{and}\quad
		\delta_k\in (0,\eps_0/2)
	\end{align}
	such that
	\begin{equation}\label{d3}
		H(x,t)=F(x,t) \geq M_k |t|^{p_B^-}, \
	\end{equation}
	for a.\,a.\,$x\in B$ and for all $|t|<\delta_k$.

	Let $r_k\in \left(0,\min\{1,\beta_k^{-1}\delta_k\}\right)$. Then, from \eqref{d2} we have
	\begin{equation}\label{d4}
		\|u\|<1
		\quad\text{and}\quad
		\|u\|_{\infty}\leq \beta_k^{-1}r_k<\delta_k<\frac{\eps_0}{2}\quad \text{for all }u\in X\cap S_{r_k},
	\end{equation}
	where $S_{r_k} =\left\{ u \in X\,:\,\, \|u\|=r_k\right\}$. Utilizing \eqref{Est.hat.M0}, \eqref{d3} and then \eqref{d4} with noticing $\operatorname{supp} (u)\subset B$ we obtain
	\begin{align*}
		J(u)&=\widehat{M}_0\left(\int_\Omega \mathcal{A}(x,\nabla u)\,\diff x\right)-\int_\Omega H(x,u)\,\diff x
		\leq M(t_0)\int_\Omega \mathcal{A}(x,\nabla u)\,\diff x-\int_{\Omega} F(x,u)\,\diff x\\
		&\leq \frac{M(t_0)}{p^-} \int_{B} \left[|\nabla u|^{p(x)}+\mu(x)|\nabla u|^{q(x)}\right]\,\diff x-M_k \int_{B}|u|^{p_B^-}\,\diff x
	\end{align*}
	for all  $u\in X_k \cap S_{r_k}$. Invoking \eqref{estimate-modular2} and \eqref{d2} we infer from the last inequality that
	\begin{align*}
		J(u)
		&\leq\frac{M(t_0)}{p^-}\|u\|^{p_B^-}-M_k\|u\|_{p_B^-}^{p_B^-}
		\leq \frac{M(t_0)}{p^-}\|u\|^{p_B^-}-M_k (\alpha_k^{-1}\|u\|)^{p_B^-}
		=\left(\frac{M(t_0)}{p^-}-M_k\alpha_k^{-p_B^-}\right) r_k^{p_B^-}.
	\end{align*}
	Thus, we obtain $\sup_{X_k \cap S_{r_k}}J(u)<0$ due to \eqref{assumptions}. Hence, $J$ satisfies \textnormal{(J2)}.

	Applying Theorem \ref{Kaj.2005} we find a sequence of critical points $\{v_k\}_{k\in\N}$ of $J$ satisfying $J(v_k)<0$ for all $k\in\N $ and $\|v_k\| \to 0$ as $k \to \infty$. Therefore,  $ v_k$ are nontrivial solutions of problem \eqref{problem'}, which can be rewritten as
	\begin{equation*}
			- \operatorname{div}A(x,\nabla u)= g(x,u)\quad\text{in } \Omega,
			\qquad u= 0 \quad \text{on } \partial\Omega,
	\end{equation*}
	where
	\begin{align*}
		g(x,u):=\frac{h(x,u)}{M_0\left(\displaystyle \int_\Omega \mathcal{A}(x,\nabla u)\diff x\right)}
		\quad\text{with}\quad |g(x,t)|\leq \dfrac{\|f_0\|_{\infty}}{m_0}
	\end{align*}
	for a.\,a.\,$x\in\Omega$ and for all $t\in\R$. According to Theorem  4.2 and Proposition 3.7 of the authors \cite{Ho-Winkert-2022}, we also have that $\|v_k\|_{\infty}\to 0$ as $k\to\infty$. Hence $\{v_k\}_{k\geq k_0} $ for some $k_0\in\N$ are solutions to our original problem \eqref{problem} and satisfy $\|v_k\|+\|v_k\|_{\infty}\to 0$ as $k\to\infty$. This finishes the proof.
\end{proof}

\section*{Acknowledgment}
The first author was partially supported by the National Research Foundation of Korea (NRF) grant funded by the Korea government (MSIT) (grant No. 2022R1A4A1032094).

\end{document}